\documentclass[10pt,a4paper,reqno]{amsart} 

\usepackage{amsmath}
\usepackage{mathtools}
\usepackage{amssymb}
\usepackage{graphicx}
\usepackage{color}
\usepackage{latexsym}
\usepackage[utf8]{inputenc}
\usepackage[T1]{fontenc}
\usepackage{comment}
\usepackage{stmaryrd}
\usepackage{bbm}

\newcommand{\zs}{{\mathbb Z}} 
\newcommand{\qs}{{\mathbb Q}}  
\newcommand{\cs}{{\mathbb C}} 

\newcommand{\eps}{\varepsilon}

\newcommand{\mb}{\bar m}

\newcommand{\Sn}{{\mathfrak S}}
\newcommand{\GA}{\mathbb{A}}

\def\Pl{P_{\ell}}
\def\Pr{P_r}
\def\Ql{Q_{\ell}}
\def\Qr{Q_r}

\DeclareMathOperator{\DR}{DR}
\DeclareMathOperator{\inv}{inv}

\newcommand{\cT}{\mathcal T}

\renewcommand{\chi}{\mathbbm 1}

\newcommand{\figeps}[3]
{\begin{figure}[ht!]
\begin{center} 
\includegraphics[width=#1cm]{#2.eps}\caption{#3}\label{fig:#2} 
\end{center}
\end{figure}}

\newtheorem{Theorem}{Theorem}
\newtheorem{Proposition}[Theorem]{Proposition}

\newtheorem{Corollary}[Theorem]{Corollary}

\newtheorem{Lemma}[Theorem]{Lemma}

\newtheorem{Property}[Theorem]{Property}

\newtheorem{Definition}[Theorem]{Definition}

\newcommand{\beq}{\begin{equation}}
\newcommand{\eeq}{\end{equation}}

\newcommand{\gf}{generating function}

\newcommand{\fps}{formal power series}

\def\emm#1,{{\em #1}}

\graphicspath{{Figures/}}

\catcode`\@=11
\def\section{\@startsection{section}{1}%
 \z@{.7\linespacing\@plus\linespacing}{.5\linespacing}%
 {\normalfont\bfseries\scshape\centering}}

\def\subsection{\@startsection{subsection}{2}%
  \z@{.5\linespacing\@plus\linespacing}{.5\linespacing}%
  {\normalfont\bfseries\scshape}}

\def\subsubsection{\@startsection{subsubsection}{3}%
 \z@{.5\linespacing\@plus\linespacing}{-.5em}
  {\normalfont\bfseries\itshape}}
\catcode`\@=12

\def\qed{$\hfill{\vrule height 3pt width 5pt depth 2pt}$}
\def\qee{$\hfill{\Box}$}

\def\cT{\mathcal{T}}
\def\cTn{\cT_n}
\def\cTnm{\cT_n^{(m)}}

\newcommand{\spacebreak}
{\begin{displaymath} \triangleleft \; \lhd \;
\diamond \; \rhd \; \triangleright
  \end{displaymath}}

\begin{document}
\title
[The number of intervals in the $m$-Tamari lattices]
{The number of intervals in the $m$-Tamari lattices}

\author[M. Bousquet-M\'elou]{Mireille Bousquet-M\'elou}
\author[\'E. Fusy]{\'Eric Fusy}
\author[L.-F. Préville-Ratelle]{Louis-François Préville-Ratelle}

\address{MBM: CNRS, LaBRI, Universit\'e Bordeaux 1, 
351 cours de la Lib\'eration, 33405 Talence, France}
\email{mireille.bousquet@labri.fr}
\address{\'EF: CNRS, LIX, \'Ecole Polytechnique, 91128 Palaiseau Cedex,
France}
\email{fusy@lix.polytechnique.fr}
\address{LFPR: LACIM, UQAM, C.P. 8888 Succ. Centre-Ville, Montréal H3C 3P8, Canada}
\email{preville-ratelle.louis-francois@courrier.uqam.ca}
%

%
\thanks{LFPR is partially supported by
 a
CRSNG grant.
\'EF and LFPR are supported by the European project
ExploreMaps -- ERC StG 208471}


\dedicatory{\large{To Doron Zeilberger, on the occasion of his 60th birthday}}

\begin{abstract}
An {$m$-ballot path} of size $n$ is a path  on the square grid
consisting of north and east steps, starting at
$(0,0)$,  ending at $(mn,n)$, and never going below the line
$\{x=my\}$. The set of these paths can be equipped with a lattice structure,
called the $m$-Tamari lattice and denoted by $\cTn^{(m)}$, which generalizes the usual Tamari
lattice $\cTn$ obtained when $m=1$. We prove that the number of intervals in
this lattice is 
$$
\frac {m+1}{n(mn+1)} {(m+1)^2 n+m\choose n-1}.
$$
This formula was recently conjectured by Bergeron  in
connection with the
study of diagonal
coinvariant spaces. The case $m=1$ was proved a few years ago
by Chapoton. Our proof is based on a recursive description of
intervals, which translates into a functional equation satisfied by
the associated \gf. The solution
of this equation is an algebraic series, obtained by a guess-and-check
approach. 
Finding a bijective proof remains an open problem.
\end{abstract}

\date{\today}
\maketitle


\section{Introduction}
A \emph{ballot path} of size   $n$ 
is a path on the square lattice, consisting of north and east steps, starting at
$(0,0)$,       
ending at $(n,n)$, and never going below the diagonal $\{x=y\}$. There
are three standard ways, often named after Stanley, Kreweras and Tamari, 
to endow the set of ballot paths of size $n$ with a lattice
structure (see~\cite{friedman-tamari,HT72,kreweras}, and~\cite{BeBo07}
or~\cite{knuth4} for a
survey). We focus here on the \emph{Tamari
  lattice} $\cTn$, which, as detailed in the following proposition,  is
conveniently  described by the associated covering relation. See
Figure~\ref{fig:push_Walk} for an illustration.
\figeps{14}{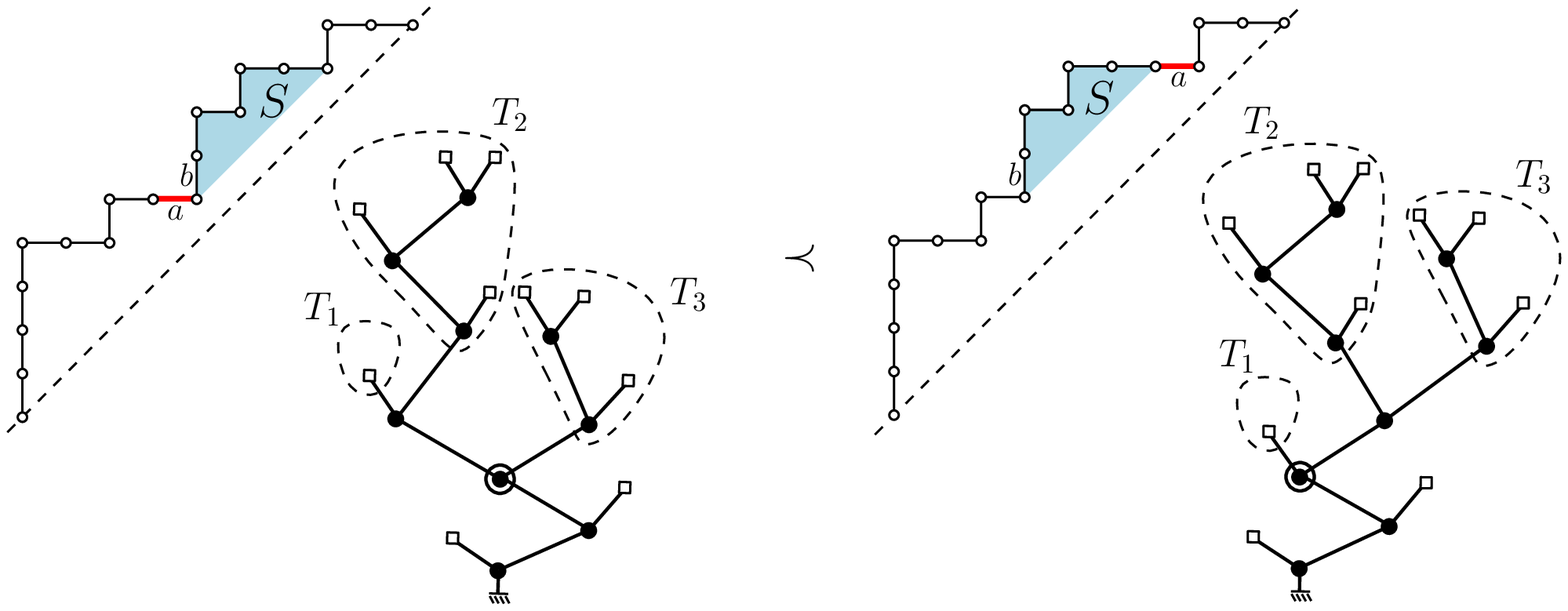}{A covering relation in the Tamari
  lattice, shown on ballot paths and binary trees. The path encodes
  the postorder of the tree 
(apart from the first leaf).}

\begin{Proposition}{\bf\cite[Prop.~2.1]{BeBo07}}
 Let $P$ and $Q$ be two ballot paths of size $n$. Then $Q$ covers
 $P$ in the Tamari lattice $\cTn$ if and only if there exists in  $P$
 an east step $a$, followed by a north step $b$, such that $Q$ is obtained from $P$ by swapping $a$ and $S$,
 where $S$ is the shortest factor of $P$ that begins with $b$ and is
 a (translated) ballot path.
\end{Proposition}

Alternatively, the Tamari lattice  $\cTn$ is often described in terms of
rooted binary trees. The covering relation amounts to  a re-organization of three
subtrees, often called \emm rotation, (Figure~\ref{fig:push_Walk}). The equivalence between the two
descriptions is obtained by reading the tree in postorder, and encoding
each leaf (resp. inner node) by a north (resp. east) step (apart from
the first leaf, which is not encoded). We refer
to~\cite[Sec.~2]{BeBo07} for details. The Hasse diagram of the
lattice $\cTn$ is the 1-skeleton of the \emm associahedron,, or \emm Stasheff
polytope,~\cite{casselman}.

A few years ago, Chapoton~\cite{ch06}  proved that the number of \emm
intervals, in 
$\cT_n$ (\emm i.e.,, pairs $P,Q\in\cTn$ such that $P\leq Q$) is
$$
\frac 2{n(n+1)} {4n+1 \choose n-1}.
$$
He observed that this number is known to count 3-connected
planar triangulations on $n+3$ vertices~\cite{tutte-triang}. Motivated by this result, Bernardi and
Bonichon found a beautiful  bijection between Tamari intervals and
triangulations~\cite{BeBo07}.
This bijection is in fact a restriction
of a more general bijection between intervals in the
Stanley lattice and \emm Schnyder woods,. A further restriction leads to the enumeration of
intervals of the Kreweras lattice.

{\begin{figure}[b!]
\begin{center} 
\includegraphics[width=12cm]{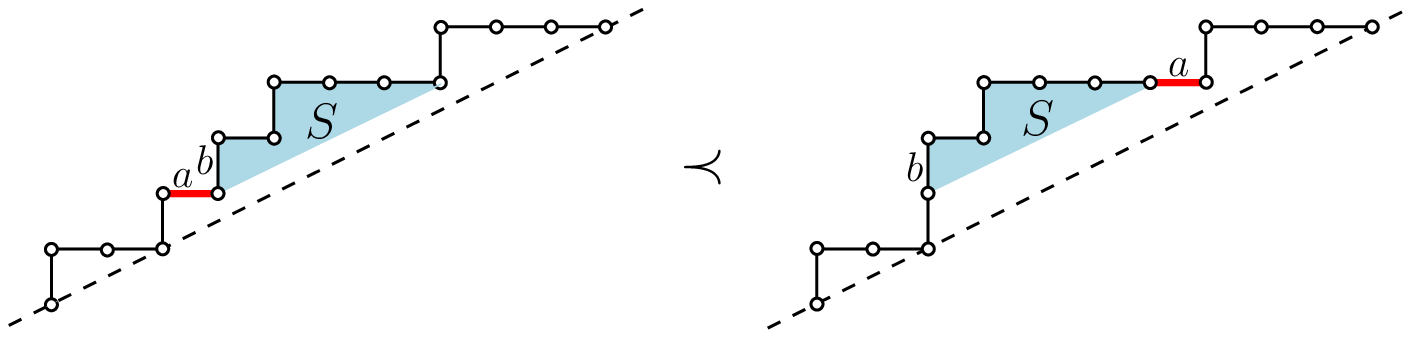}\caption{The relation $\prec$  between
  $m$-ballot paths ($m=2$).}\label{fig:push_mWalk} 
\end{center}
\end{figure}}

\medskip
In this paper, we study a generalization of the Tamari lattices to
\emm $m$-ballot paths, due to Bergeron, and count the intervals of these
lattices. Again, a remarkably simple formula holds
(see~\eqref{number}). As we explain 
below, this formula was first conjectured by F.~Bergeron,
in connection with the study of coinvariant spaces.

 An \emph{$m$-ballot path} of size $n$ is a path  on the square grid
consisting of north and east steps, starting at
$(0,0)$,  ending at $(mn,n)$, and never going below the line
$\{x=my\}$. It is a classical  
exercice to show that  there are $\frac 1{mn+1}{(m+1)n \choose n}$
such paths~\cite{dvoretzky}.  
Consider the following relation $\prec$ on $m$-ballot paths, illustrated in
Figure~\ref{fig:push_mWalk}.
\begin{Definition}
\label{def-m-tamari}
Let $P$ and $Q$ be two $m$-ballot paths of size $n$.
 Then $P \prec Q$ if  there exists in  $P$
 an east step $a$, followed by a north step $b$, such that $Q$ is
 obtained from $P$ by swapping $a$ and $S$, 
 where $S$ is the shortest factor of $P$ that begins with $b$ and is
 a (translated) $m$-ballot path.
\end{Definition}
As we shall see, the transitive closure of $\prec$ defines a lattice
on $m$-ballot paths of size $n$. We call it the 
$m$-Tamari lattice of size $n$,
and denote it by $\cTnm$. Of course, $\cTn^{(1)}$ coincides with
$\cTn$. See Figure~\ref{fig:lattice_ex} for examples.
The main result of this paper is a closed form
expression for the number $f_n^{(m)}$ of intervals  in $\cT_{n}^{(m)}$: 
\beq\label{number}
f_n^{(m)}=\frac {m+1}{n(mn+1)} {(m+1)^2 n+m\choose n-1}.
\eeq

\begin{figure}
\begin{center}
\includegraphics[height=9cm]{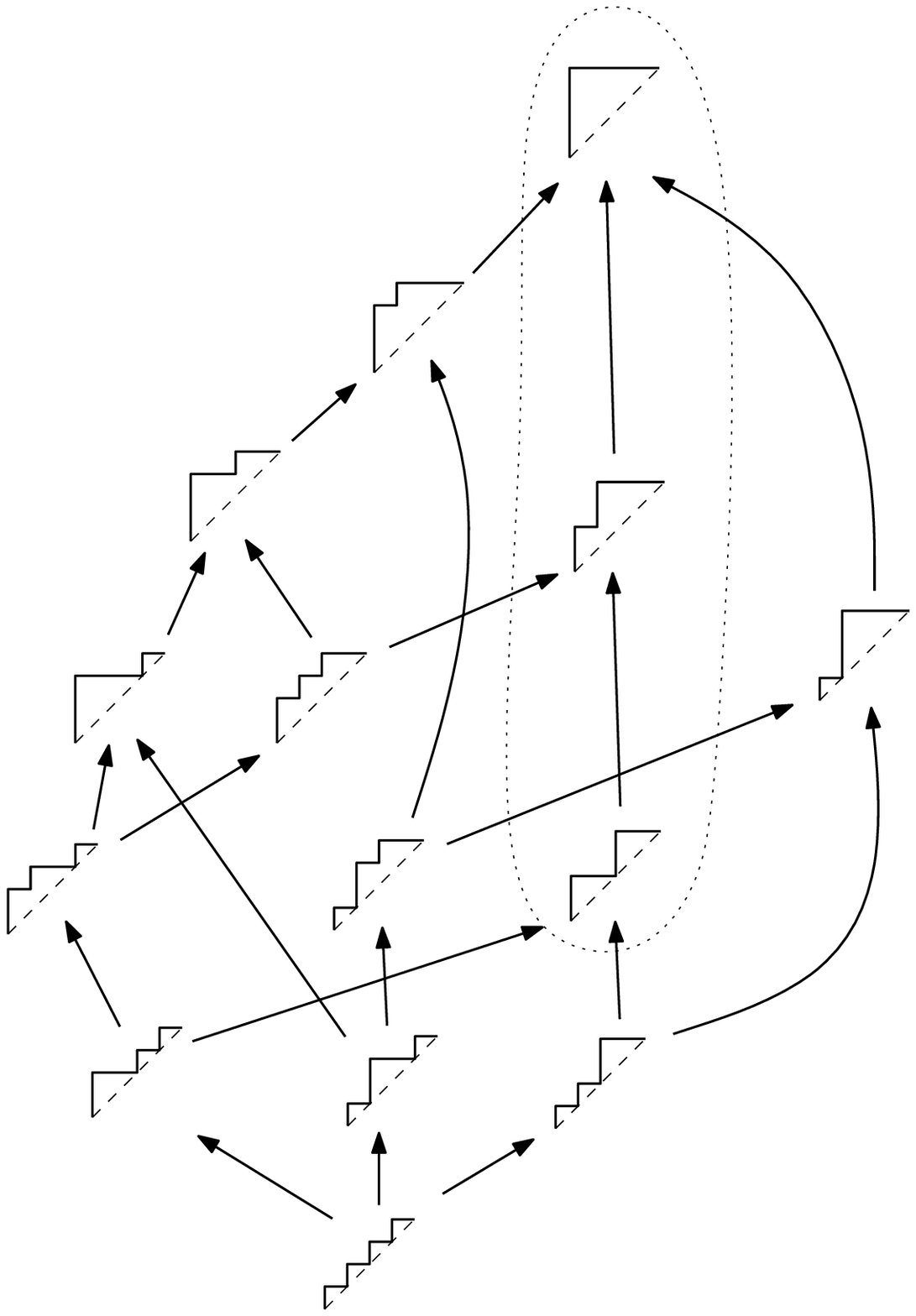}\hspace{1cm}\includegraphics[height=9cm]{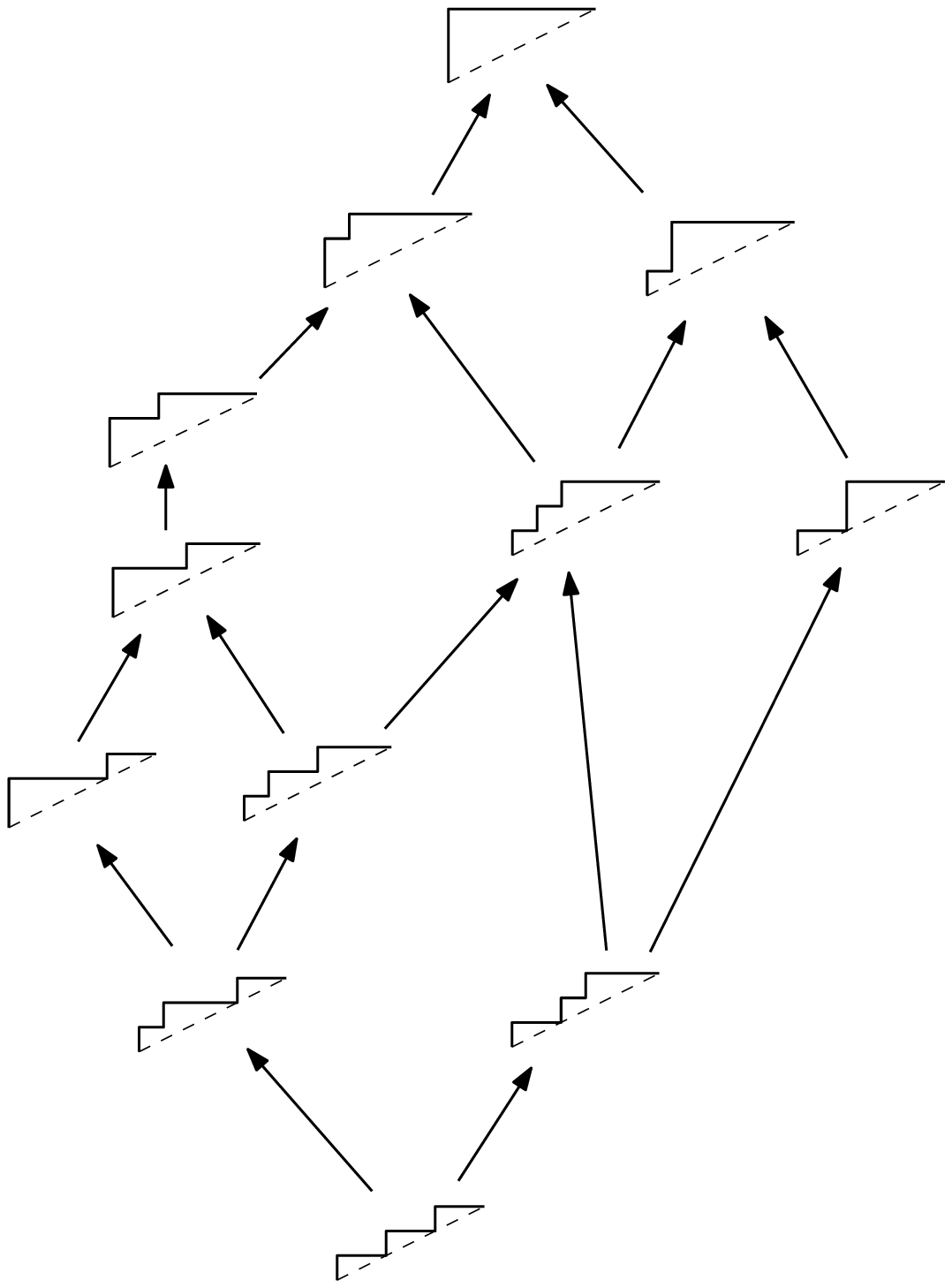}
\end{center}
\caption{The $m$-Tamari lattice $\cT_{n}^{(m)}$ for  $m=1$ and $n=4$ (left)
and for $m=2$ and $n=3$ (right). 
The three walks surrounded by a line in $\cT_{4}^{(1)}$ form a lattice
that is isomorphic to $\cT_{2}^{(2)}$. This will be generalized  in Section~\ref{sec:eq}.}
\label{fig:lattice_ex}
\end{figure}

The first step of  our proof establishes that  $\cTnm$ is in fact
isomorphic to a
sublattice (and more precisely, an upper ideal) of $\cT_{mn}$. 
We then proceed with a recursive description of the intervals of
$\cTnm$, which 
translates into a functional equation for the associated \gf\
(Section~\ref{sec:eq}, Proposition~\ref{prop:eq}). This
\gf\ 
keeps track of the size of the paths, but also of  a \emm
catalytic parameter,\footnote{This terminology is due to
  Zeilberger~\cite{zeil-umbral}.} that is needed to write the
equation. This parameter is the number of contacts of the lower path
with the line $\{x=my\}$. A
general theorem asserts that the solution of the equation is
algebraic~\cite{mbm-jehanne}, and gives a systematic procedure to solve it for
small
values of $m$. However, for a generic value of $m$, we have to resort
to a guess-and-check approach to solve the equation
(Section~\ref{sec:sol}, Theorem~\ref{thm:series}). 
We enrich our enumeration by taking
into account the \emm initial rise, of the upper path, that is, the
length of its initial run of north steps. We obtain an unexpected
symmetry result: the joint distribution of the number of contacts of
the lower path (minus one) and the initial rise of the upper path is
symmetric. Section~\ref{sec:final} presents  comments and questions.

\medskip
To conclude this introduction, we  describe the algebraic
problem that led  Bergeron  to conjecture~\eqref{number}.

Let $X=(x_{i,j})_{^{1 \leq i \leq \ell}_{1 \leq j \leq n}}$ be a matrix
of variables, for some positive integers $\ell,n \geq 1$. We call each
line of $X$ a \emm set of variables,. 
Let $\mathbb{C}[X]$ be the ring of polynomials in the variables of
$X$. The symmetric group $\Sn_n$ acts as a representation
on $\mathbb{C}[X]$ 
 by permuting the columns of $X$. That is,  if $\sigma \in \Sn_n
$ and $f(X) \in \mathbb{C}[X]$, then
$$
\sigma(f(X))=f(\sigma(X))=f((x_{i,\sigma(j)})_{^{1 \leq i \leq \ell}_{1 \leq j \leq n}}).$$
We consider the ideal $I$ of $\cs[X]$ generated by $\Sn_n$-invariant
polynomials having no constant term.  The
 quotient ring $\cs[X]/I$ is (multi-)graded because 
$I$ is (multi-)homogeneous, and is a representation of $\Sn_n$ because $I$ is
invariant under the action of $\Sn_n$. We focus on the dimension of
this quotient ring, and to the dimension of the sign 
subrepresentation. We denote by $W^ \varepsilon$  the sign
subrepresentation of a representation $W$.

Let us begin with the classical case of a single set of variables. When
$X=[x_1,\ldots ,x_n]$, we consider the \emm coinvariant space, $R_n$,
defined by
\begin{equation*}
R_n = \mathbb{C}[X] \Big / \big\langle \big\{ \sum_{i=1}^n x_i^{r} \big| r \geq 1 \big\} \big\rangle ,
\end{equation*}
where $\big\langle S \big\rangle$ denotes the ideal generated by
the set $S$.
It is known~\cite{Artin} that $R_n$ is isomorphic to the regular
representation of $\Sn_n$. In particular, $\dim(R_n)=n!$ and
$\dim(R_n ^\varepsilon)=1$. 
There exist explicit bases of $R_n$ indexed by permutations.

Let us now move to two sets of variables. In the early nineties, Garsia and
Haiman introduced an analogue of $R_n$  for $X 
= \begin{bmatrix} x_1&\ldots&x_n\\ y_1&\ldots&y_n \end{bmatrix}$, and 
called it the \emm diagonal coinvariant space,~\cite{HaiConj}: 
\begin{equation*}
\DR_{2,n} = \mathbb{C}[X] \Big / \big\langle \big\{ \sum_{i=1}^n x_i^{r} y_i^{t} \big| r+t \geq 1 \big\} \big\rangle .
\end{equation*}
About ten years later, using advanced algebraic geometry~\cite{HaimanPreu},
Haiman settled several conjectures of~\cite{HaiConj} concerning this
space,  proving in particular that   
\beq\label{2sets}
\dim({\DR_{2,n}} ) = (n+1)^{n-1} \quad \hbox{ and } \quad
\dim(\DR_{2,n}^{\ \varepsilon} ) =\frac{1}{n+1} \binom{2n}{n} .
\eeq
He  also studied  an extension of $\DR_{2,n}$ involving an integer
parameter $m$ and the ideal $\mathcal A $ generated by \emm
alternants,
\cite{garsia-haiman-lagrange,MR2115257}:
$$
\mathcal A = \big\langle \big\{ f(x) \big| \sigma(f(X))= (-1)^{\inv(\sigma)}f(X) , \forall \sigma \in \Sn_n \big\} \big\rangle .
$$
There is a natural action of $\Sn_n$ on the quotient space
$\mathcal{A}^{m-1} \big{/} \mathcal{J} \mathcal{A}^{m-1}$. Let us
twist this action by the $(m-1)^{\rm st}$ power of the sign
representation $\eps$: this gives rise to spaces
$$
\mathcal{DR}_{2,n}^{m} := {\eps}^{m-1}  \otimes  \mathcal{A}^{m-1} \big{/} \mathcal{J} \mathcal{A}^{m-1},
$$
so that $\mathcal{DR}_{2,n}^{1}=\mathcal{DR}_{2,n}$.
Haiman~\cite{HaimanPreu,MR2115257} generalized~\eqref{2sets} by   proving
$$
\dim({\DR_{2,n}^m}) = (mn+1)^{n-1} \quad \hbox{and } \quad 
\dim(\DR_{2,n}^{ m\ \varepsilon} ) = \frac{1}{mn+1}
\binom{(m+1)n}{n}.
$$
Both dimensions have simple combinatorial interpretations: we
recognize in the latter
 the number of $m$-ballot paths of size $n$, and the former  is
the number of \emm $m$-parking functions, of size $n$ (these functions
can be described as $m$-ballot paths of size $n$ in which the north
steps are labelled from $1$ to $n$ in such a way the labels increase
along each run of north steps; see \emm e.g.,~\cite{yan}). However, it
is still an open problem to find bases of ${\DR_{2,n}^m}$ or
$\DR_{2,n}^{ m\ \varepsilon}$ indexed by these simple combinatorial
objects.

For  $\ell \geq 3$, the spaces $\DR_{\ell,n}$ and their generalization
${\DR_{\ell,n}^m}$ can be defined similarly. Haiman 
 explored  the dimension of $\DR_{\ell,n}$  and 
$\DR_{\ell,n}^{\ \varepsilon}$. For $\ell=3$, 
he observed in~\cite{HaiConj} that, for small values of $n$,
\begin{equation*}
\dim({\DR_{3,n}}) = 2^{n} (n+1)^{n-2} \quad \hbox{and } \quad 
\dim(\DR_{3,n}^{\ \varepsilon} ) = \frac 2{n(n+1)} {4n+1 \choose n-1}.
\end{equation*}
Following discussions with Haiman, Bergeron came up with conjectures
that directly imply the following generalization (since $\DR_{3,n}^1$
coincides with $\DR_{3,n}$):
\begin{equation*}
\dim({\DR_{3,n}^m}) = (m+1)^{n} (mn+1)^{n-2} \quad \hbox{and } \quad 
\dim(\DR_{3,n}^{m \ \varepsilon} ) = \frac{m+1}{n(mn+1)}
\binom{(m+1)^2 n + m}{n-1}. 
\end{equation*}
Both conjectures are still wide open. 

A much simpler problem consists
in asking whether these dimensions again have a simple combinatorial
interpretation. 
Bergeron, starting from the sequence $\frac 2{n(n+1)} {4n+1 \choose n-1}$, found in Sloane's Encyclopedia
that this number counts, among others, certain ballot related objects,
namely  intervals in the Tamari lattice~\cite{ch06}. 
From this observation, 
and the role played by $m$-ballot paths for two sets of variables,
he was led to introduce  the $m$-Tamari lattice $\cT_{n}^{(m)}$, and conjectured
that $\frac{m+1}{n(mn+1)}
\binom{(m+1)^2 n + m}{n-1}$ is the number of intervals in
this lattice. This is the conjecture we prove in this paper. Another
of his conjectures is that  $(m+1)^{n} (mn+1)^{n-2} $ is the number of
Tamari intervals where the larger path is ``decorated'' by an $m$-parking
function~\cite{bergeron-preville}. 
This is proved in~\cite{mbm-chapuy-preville,mbm-chapuy-preville-prep}. 
%

\section{A functional equation for the \gf\ of intervals}
\label{sec:eq}
The aim of this section is to describe a recursive decomposition of $m$-Tamari
intervals, and to translate it into a functional equation satisfied by
the associated \gf\ (Proposition~\ref{prop:eq}). There are two main
tools: 
\begin{itemize}
\item we  prove that $\cTn^{(m)}$ can be seen as an upper ideal of
the usual Tamari lattice $\cT_{mn}$,
\item we give a simple criterion
to decide when two paths of the Tamari lattice are comparable.
\end{itemize}
 
\subsection{An alternative description of the $m$-Tamari  lattices}
\label{sec:reform}
Our first transformation is totally harmless: we apply a 45 degree
rotation to 1-ballot paths to transform them into \emm Dyck paths,.
A Dyck path of size $n$ consists  of steps $(1,1)$ (up steps) 
and  steps $(1,-1)$ (down steps), starts at $(0,0)$, ends
at $(0,2n)$ and never goes below the $x$-axis. 

We now introduce some terminology, and use it to rephrase the description of the
(usual) Tamari lattice $\cTn$.
Given a Dyck path $P$, and  an up step $u$ of $P$,  the shortest portion of $P$ that
starts with $u$ and forms a (translated) Dyck path is called the
\emm excursion of $u$ in $P$.,
We say that $u$ and the final step of its excursion \emm match,
each other. Finally, we say that $u$ has \emm rank, $i$ if it is the
$i^{\hbox{\small th}}$ up step of $P$.

 Given two Dyck paths $P$ and $Q$  of size $n$,  $Q$ covers
 $P$ in the Tamari lattice $\cTn$ if and only if there exists in  $P$
 a down step $d$, followed by an up step $u$, such that $Q$ is obtained from $P$ by swapping $d$ and $S$,
 where $S$ is the excursion of $u$ in $P$.
This description implies the following  property~\cite[Cor.~2.2]{BeBo07}.

\begin{Property}\label{fact:above}
If $P\leq Q$ in $\cT_n$ then $P$ is  below $Q$. That is, for
$i\in[0..2n]$, the ordinate of the vertex of $P$ lying at abscissa $i$
is at most  the ordinate of the vertex of $Q$ lying at abscissa $i$.
\end{Property}

Consider now an $m$-ballot path of size $n$, and  replace each north step
 by a sequence of $m$ north steps. This gives a 1-ballot path of size $mn$, and
 thus, after a rotation, a Dyck path. In this path, for each
 $i\in[0..n-1]$, 
the up steps of ranks $mi+1,\ldots,m(i+1)$
are consecutive. We call the Dyck paths satisfying this property
\emph{$m$-Dyck paths}. Clearly, $m$-Dyck paths of size $mn$ are in
one-to-one correspondence with $m$-ballot paths of size $n$.
Consider now the relation $\prec$ of Definition~\ref{def-m-tamari}: once
reformulated in terms of Dyck paths, it becomes a covering relation in the
(usual) Tamari lattice (Figure~\ref{fig:push_mPath}). Conversely, it
is easy to check that, if $P$ 
is an $m$-Dyck path and $Q$ covers $P$ in the usual Tamari lattice,
then $Q$ is also an $m$-Dyck path, and the $m$-ballot paths
corresponding to $P$ and $Q$ are related by $\prec$. We have thus
proved the following result.

\figeps{14}{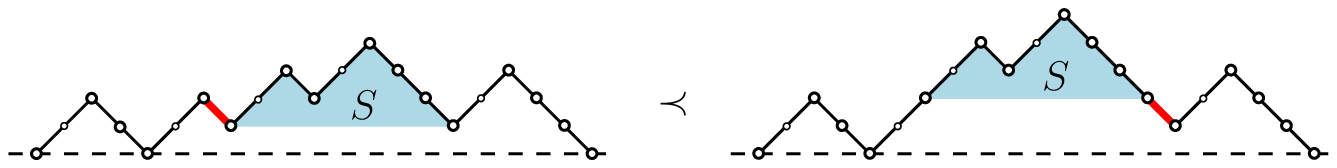}{The relation $\prec$  of Figure~\ref{fig:push_mWalk}
reformulated  in terms of $m$-Dyck
paths ($m=2$).}

\begin{Proposition}
  The transitive closure of the relation $\prec$ defined in
  Definition~\ref{def-m-tamari} is a lattice on $m$-ballot paths of
  size $n$. This lattice is isomorphic to the sublattice of the Tamari
  lattice $\cT_{mn}$ consisting of the elements that are larger than
  or equal to the Dyck path $u^m d^m \ldots u^m d^m$. The relation
  $\prec$ is the covering relation of this lattice.
\end{Proposition}

\noindent{\bf Notation.} From now on, we only consider Dyck paths. We
denote by $\cT$ the set of Dyck paths, and by $\cTn$ the Tamari lattice of
Dyck paths of length $n$. By $\cT^{(m)}$ we mean the set of $m$-Dyck
paths, and by $\cTn^{(m)}$ the Tamari lattice of $m$-Dyck paths of
size $mn$. This lattice is a sublattice of $\cT_{mn}$. Note that
$\cT^{(1)}= \cT$ and  $\cTn^{(1)}=\cTn$.

\subsection{Distance functions}
Let $P$ be a Dyck path of size $n$. For an up step $u$ of $P$, we
denote by   $\ell_P(u)$ the size of the excursion of $u$ in $P$.
The function $D_P:[1..n]\to[1..n]$ defined by $D_P(i)=\ell(u_i)$,
where $u_i$ is the $i^{\hbox{\small {th}}}$ up step of $P$,  is called the \emm
distance function, of $P$. It will sometimes be convenient to see
$D_P$ as a vector $(\ell(u_1), \ldots, \ell(u_n))$ with $n$
components. 
In particular, we will compare distance functions component-wise.
The main result of this subsection is a description of the Tamari
order in terms of distance functions. This simple characterization
seems to be new.
\begin{Proposition}\label{prop:distance}
 Let $P$ and $Q$ be two paths in the Tamari lattice $\cTn$. Then
 $P\leq Q$ if and only if $D_P\leq D_Q$.
\end{Proposition}

In order to prove this, we first describe the relation between  the
distance functions of two paths related by a covering relation.

\begin{Lemma}\label{lem:Dpush}
Let $P$ be a Dyck path, and $d$ a down step of $P$ followed by an up step
$u$. Let $S$ be the excursion of  $u$ in $P$, and let $Q$ be the path
obtained from $P$  by swapping $d$ and $S$.
Let $u'$ be the up step matched with $d$ in $P$, and $i_0$ the rank  
of $u'$ in $P$. 
Then $D_{Q}(i)=D_P(i)$ 
for each $i\neq i_0$ and $D_{Q}(i_0)=D_P(i_0)+\ell_P(u)$.  
\end{Lemma}

\figeps{14}{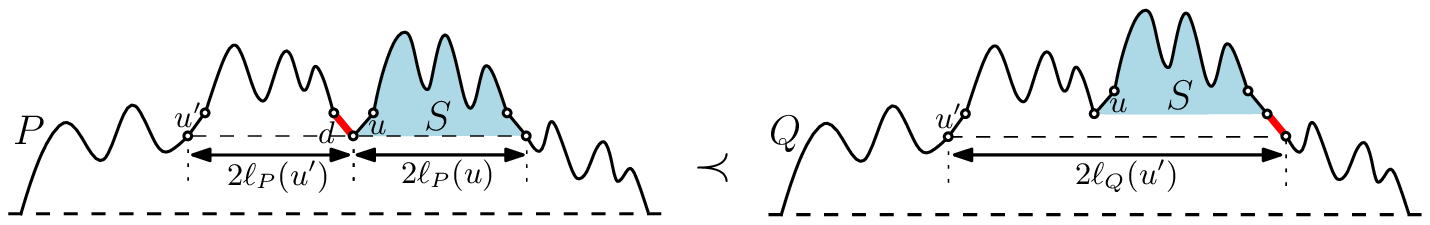}{How the distance function changes in a
  covering relation.}

This  lemma is easily proved using Figure~\ref{fig:proof_push}. It already implies
that $D_P\leq D_Q$ if $P\leq Q$.
 The next lemma establishes the reverse implication, thus concluding
 the proof of Proposition~\ref{prop:distance}.

\begin{Lemma}\label{lem:leqTamari}
Let $P$ and $Q$ be two Dyck paths of size $n$ 
such that $D_P\leq D_Q$. Then $P\leq Q$ in the Tamari lattice $\cTn$.
\end{Lemma}
\begin{proof}
Let us first prove, by induction on the size,  that  $P$ is below
$Q$ (in the sense of Property~\ref{fact:above}). This is clearly true
if $n=0$, so we assume $n>0$.

Let $u$ be the first up step (in $P$ and $Q$). Note that
$\ell_P(u)=D_P(1)\leq D_Q(1)=\ell_Q(u)$.  
Let $P'$ (resp. $Q'$) be the path obtained from $P$ (resp. $Q$) by
contracting $u$ and the down step matched with $u$. 
Observe that  $D_{P'}$ is obtained by deleting the first component of
$D_P$, and similarly for  $D_{Q'}$ and $D_Q$. Consequently  $D_{P'}\leq
D_{Q'}$, and hence 
by the induction hypothesis, $P'$ is below $Q'$. 
Let us consider momentarily Dyck paths as functions, and write $P(i)=j$ if the
vertex of $P$ lying at abscissa $i$ has ordinate $j$. 
Note that $P(i)=P'(i-1)+1$ 
 for $1\leq i< 2\ell_P(u)$, and
$P(i)=P'(i-2)$ for $2\ell_P(u)\leq i\leq 2n$.  
Similarly $Q(i)=Q'(i-1)+1$ for $1\leq i< 2\ell_Q(u)$, and
$Q(i)=Q'(i-2)$ for $2\ell_Q(u)\leq i\leq 2n$.  
Since $\ell_P(u)\leq \ell_Q(u)$ and $P'(i)\leq Q'(i)$ for $0\leq i\leq
2n-2$, one easily checks that 
$P(i)\leq Q(i)$ for $0\leq i\leq 2n$, so that $P$ is below $Q$.

\medskip

In order to prove that $P \leq Q$, we proceed by induction on
$||D_{P}-D_Q||$, where $||(x_1,  \ldots, x_n)||= |x_1|+ \cdots
+|x_n|$. 
If $D_P=D_Q$ then $P=Q$, because $P$ is below $Q$ and $Q$ is below
$P$.
So let us assume that $D_P\not=D_Q$.
Let $i$ be minimal such that $D_P(i)<D_Q(i)$. We claim that
$P$ and $Q$ coincide at least  until their up step of rank
$i$. Indeed,  since $P$ lies below $Q$, the paths
$P$ and $Q$ coincide up to some 
abscissa, and then we find a down step $\delta$ in $P$ but an up step in
$Q$. Let $j$ be the rank of the up step that matches $\delta$ in
$P$. This up step belongs also to $Q$, and, since $\delta \not \in Q$,
we have
 $D_P(j)< D_Q(j)$. Hence $j\ge i$ by minimality of $i$, and  $P$
 and $Q$ coincide at least until their up step of rank $i$, which we
 denote by $u$. Let $d$ be the down step matched with $u$ in
 $P$ (Figure~\ref{fig:snotdescending}). Since $D_P(i)<D_Q(i)$, the step $d$ 
is \emm not, a step of $Q$. The step of $Q$ located at the same
abscissa as $d$ ends strictly higher than $d$, and in particular, at a
positive ordinate. Hence $d$ is not the final step of $P$. Let $s$ be the
step following  $d$ in $P$.

\figeps{12}{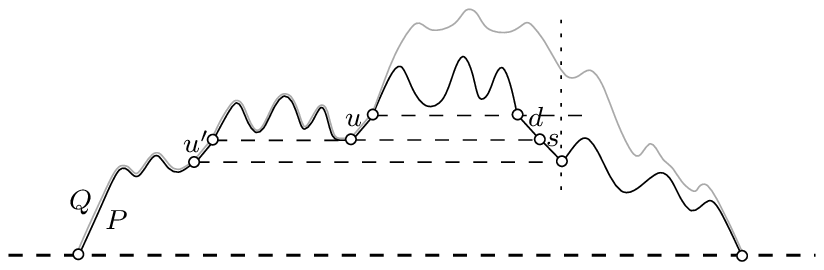}{Why $s$ cannot be descending.}

Let us prove \emm ad  absurdum, that $s$ is an up step. Assume $s$ is
down. Then $s$ is matched in $P$ with an up step $u'$ of rank
$j<i$ (Figure~\ref{fig:snotdescending}). Hence $u'$ belongs to  $Q$ and has rank
$j$ in $Q$. Since $s$ cannot 
belong to $Q$, this implies that $D_P(j)<D_{Q}(j)$, which contradicts
the minimality of~$i$.

 Hence $s$ is an up step of $P$ (Figure~\ref{fig:proof_lem}). Let  $S$ be the excursion of $s$ in $P$.  
Since $\ell_Q(u)>\ell_P(u)$ and since $Q$ is above $P$, we have
$\ell_Q(u)\geq \ell_P(u)+\ell_P(s)$,
\emm i.e.,, $D_Q(i)\geq D_P(i)+\ell_P(s)$. 
Let $P'$ be the path obtained from $P$ by swapping $d$ and $S$. Then
$P'$ covers $P$ in the Tamari lattice. By Lemma~\ref{lem:Dpush},
$D_P=D_{P'}$ except at index $i$ (the rank of $u$), where
$D_{P'}(i)=D_P(i)+\ell_P(s)$.
Since $D_P(i)+\ell_P(s)\le D_Q(i)$, we have $D_{P'}\leq D_Q$. But
$||D_{P'}-D_Q||= ||D_{P}-D_Q||-\ell_P(s)$ and  by the induction hypothesis,
$P'\leq Q$ in the Tamari lattice. 
Hence $P< P'\leq Q$, and the lemma is proved.
\end{proof}

\figeps{12}{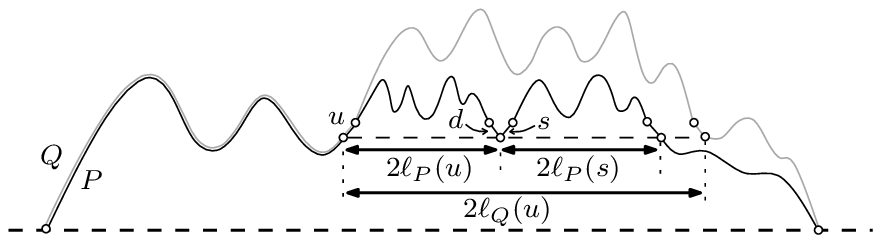}{General form of $P$ and $Q$.}

\subsection{Recursive decomposition of intervals and functional
  equation}
\label{sec:rec}
A \emph{contact} of a Dyck path $P$ is a vertex of $P$ lying on the
$x$-axis. It is \emm initial, if it is $(0,0)$.  A \emm contact, of a
Tamari interval $[P, Q]$ is a contact of the \emm lower, path $P$. The
recursive decomposition of intervals that we use makes  the number of
contacts crucial, and we  say that this parameter is \emm catalytic,. We  also
consider another, non-catalytic parameter, which we find to be
equidistributed with non-initial contacts (even more, the joint
distribution of these two parameters is symmetric). Given an $m$-Dyck
path $Q$, the length of the initial run of up steps is 
of the form $mk$; the integer $k$ is called 
the \emph{initial rise} of $Q$. The \emm initial rise, of an interval $[P,
Q]$ is the initial rise of the \emm upper, path $Q$.  The aim of this subsection is to
establish the following functional equation.
\begin{Proposition}\label{prop:eq}
For $m\ge 1$, let  $F(x)\equiv F^{(m)}(t;x)$ be the \gf\ of $m$-Tamari
intervals, where $t$ counts the size (divided by $m$)
and $x$ the number of contacts. Then
$$
F(x)=x+xt\left(F(x)\cdot  \Delta\right)^{(m)}(F(x)),
$$
where $\Delta$ is the following divided difference operator
$$
\Delta S(x)=\frac{S(x)-S(1)}{x-1},
$$
and the power $m$ means  that the operator $G(x)\mapsto F(x)\cdot \Delta
G(x)$ is applied $m$ times to $F(x)$.

More generally, if  $F(x,y)\equiv F^{(m)}(t;x,y)$ keeps track in
addition of the initial rise (via the variable $y$), we have the
following functional equation:
\beq\label{eq:Fb}
F(x,y)=x+xyt\left(F(x,1)\cdot \Delta\right)^{(m)}(F(x,y)).
\eeq
\end{Proposition}
Note that each of the above two equations defines a unique formal
power series in 
$t$ (think of extracting inductively the coefficient of $t^n$ in
$F^{(m)}(t;x)$ or $F^{(m)}(t;x,y)$).
\\
\noindent{\bf Examples}\\
1. When $m=1$, the above equation reads
\begin{eqnarray*}
  F(x,y)&=&x+xytF(x,1)\cdot \Delta(F(x,y))\\
&=& x+xyt F(x,1)\  \frac{F(x,y)-F(1,y)}{x-1}.
\end{eqnarray*}
When $y=1$, we obtain, in the terminology of~\cite{mbm-jehanne}, a
 quadratic equation with one catalytic variable:
$$
F(x)= x+xt F(x)\  \frac{F(x)-F(1)}{x-1}.
$$
2. When $m=2$,
\begin{eqnarray*}
  F(x,y)&=&x+xyt\ F(x,1)\cdot  \Delta(F(x,1)\cdot \Delta(F(x,y)))\\
&=& x+xyt\ F(x,1)\cdot \Delta\left(F(x,1) \  \frac{F(x,y)-F(1,y)}{x-1}\right)\\
&=&x+\frac{xyt}{x-1}\  F(x,1)\  \left( F(x,1) \ \frac{F(x,y)-F(1,y)}{x-1}-F(1,1)F'(1,y)\right),
\end{eqnarray*}
where the derivative is taken with respect to the variable $x$.
When $y=1$, we obtain a cubic
equation with one catalytic variable:
$$
F(x)= x+\frac{xt}{x-1} F(x)\  \left( F(x)\ 
\frac{F(x)-F(1)}{x-1}-F(1)F'(1)\right).
$$
The solution of~\eqref{eq:Fb} will be the topic of the next
section. For the moment we focus on the proof of this equation.

\medskip
We say that a vertex $q$ lies to the right of a vertex $p$ if the
abscissa of $q$ is greater than  or equal to the abscissa of $p$.
A \emph{$k$-pointed Dyck path} is a tuple $(P;p_1,\ldots,p_k)$ where $P$ is a
Dyck path and $p_1,\ldots,p_k$ are contacts of $P$ such that $p_{i+1}$
lies to the right of $p_i$, for $1\le i<k$ (note that some $p_i$'s
may coincide). Given an $m$-Dyck path $P$ of positive size,
let $u_1,\ldots,u_m$
be the initial (consecutive) up steps of $P$, and let $d_1,\ldots,d_m$
be the down steps matched with $u_1,\ldots,u_m$, respectively.
The \emph{$m$-reduction} of $P$ is the $m$-pointed Dyck path
$(P';p_1,\ldots,p_m)$ where $P'$ is obtained from $P$  
by contracting all the steps $u_1,\ldots,u_m,d_1,\ldots,d_m$, and
$p_1,\ldots,p_m$ are  the  vertices of $P'$ resulting from the contraction of
$d_1,\ldots,d_m$. It is easy to check that they are indeed contacts of
$P'$ (Figure~\ref{fig:reduction}).

\figeps{14}{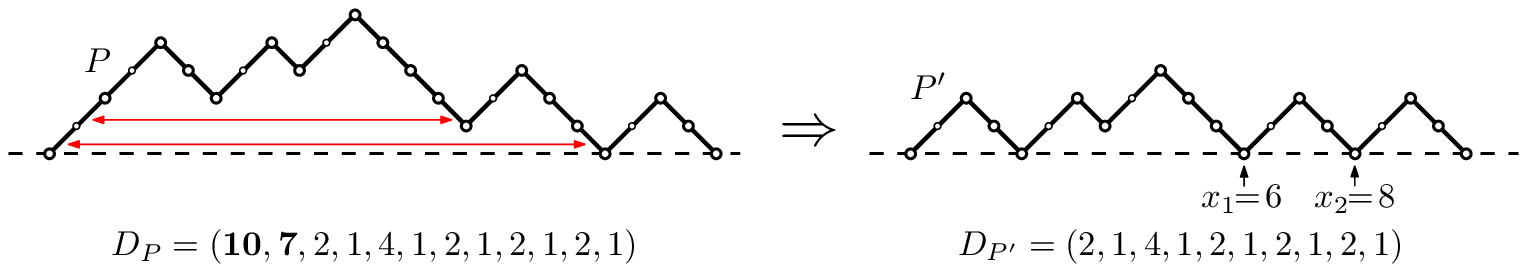}{The $m$-reduction of an $m$-Dyck
path ($m=2$).}

The map $P\mapsto(P';p_1,\ldots,p_m)$  is clearly invertible, hence
$m$-Dyck paths of size $mn$ are in bijection with $m$-pointed $m$-Dyck 
paths of size $m(n-1)$. 
Note that the non-initial contacts of $P$  correspond to the contacts
of $P'$ that lie  to the right of $p_m$. 
Note also that the  distance function $D_{P'}$ (seen as a vector with
$m(n-1)$ components) is obtained by deleting the first $m$ components of $D_P$. Conversely, denoting by $2x_i$ 
the abscissa of $p_i$, $D_P$ is obtained  by 
prepending to $D_{P'}$ the sequence
$(x_m+m,x_{m-1}+m-1,\ldots,x_1+1)$.
In view of
Proposition~\ref{prop:distance}, this gives the following recursive
characterization of intervals. 

\begin{Lemma}\label{lem:interval_reduc}
Let $P$ and $Q$ be two $m$-Dyck paths of size $mn>0$. Let
$(P';p_1,\ldots,p_m)$ and $(Q';q_1,\ldots,q_m)$ be the $m$-reductions
of $P$ and $Q$ respectively. 
Then $P\leq Q$ in $\cT_n^{(m)}$ if and only if $P'\leq Q'$ in $\cT_{n-1}^{(m)}$
and for $i\in[1..m]$, the point $q_i$ lies to the right of  $p_i$. 

 The non-initial contacts of $P$  correspond to the contacts of $P'$ located
 to the right of $p_m$.  
\end{Lemma} 
Let us call \emm $k$-pointed interval, in
$\cT^{(m)}$  a pair consisting  of two
 $k$-pointed $m$-Dyck paths $(P;p_1,\ldots,p_k)$ and
 $(Q;q_1,\ldots,q_k)$ such that $P\le Q$ and for $i\in[1..k]$, the
  point $q_i$ lies to the right of $p_i$. 
An \emph{active contact} of such a pair is a contact of $P$ lying
 to the right of $p_k$ (if $k=0$, all contacts are declared active).  For $0\le k \le m$, let us denote by
$G^{(m,k)}(t;x,y)\equiv G^{(k)}(x,y)$ the \gf\ of $k$-pointed
$m$-Tamari intervals, where $t$ counts the size (divided by $m$), $x$
the number of active contacts,  and $y$ the initial rise (we drop the
superscript $m$ since it will not vary). 
In particular, the series we are interested in is
\beq\label{eq:fun1}
F(x,y)=G^{(0)}(x,y).
\eeq
Moreover, Lemma~\ref{lem:interval_reduc} implies
\beq\label{eq:fun2}
F(x,y)=x+xyt\ \!G^{(m)}(x,y).
\eeq
We will  prove that, for $k \ge 0$,
\beq\label {Gk}
G^{(k+1)}(x,y)=F(x,1)\cdot\Delta G^{(k)}(x,y).
 \eeq
The functional equation~\eqref{eq:Fb} then follows using~\eqref{eq:fun1}
and~\eqref{eq:fun2}.

For $k\geq 0$, let $I=[P^\bullet, Q^\bullet]$ be a $(k+1)$-pointed
interval in $\cT^{(m)}$, with $P^\bullet=(P;p_1,\ldots,p_{k+1})$ and
$Q^\bullet=(Q;q_1,\ldots,q_{k+1})$
(see an illustration in Figure~\ref{fig:dec_int_bis} when $k=0$). 
Since $P$ is  below $Q$, the contact $q_{k+1}$ of $Q$ is
also a contact of $P$.  By definition of pointed intervals, $q_{k+1}$ is 
to the right of $p_1,\ldots,p_{k+1}$.
Decompose $P$ as $\Pl\Pr$ where $\Pl$ is the part of $P$
to the left of $q_{k+1}$ and $\Pr$ is the part of $P$ to
the right of $q_{k+1}$. Decompose similarly $Q$ as $\Ql\Qr$, where the
two factors meet at $q_{k+1}$.
 The distance function $D_P$ (seen as a vector)
is $D_{\Pl}$ concatenated with $D_{\Pr}$, and similarly for $D_Q$. 
In particular, $D_{\Pl}\le D_{\Ql}$ and $D_{\Pr}\le D_{\Qr}$. By
Proposition~\ref{prop:distance},  $I_r:=[\Pr, \Qr]$ is an
 interval, while
$I_\ell:= [P^\circ, Q^\circ]$, with $P^\circ=(\Pl;p_1,\ldots,p_k)$ and
$Q^\circ=(\Ql;q_1,\ldots,q_k)$, is a $k$-pointed interval.  
Its initial rise equals the initial rise of $I$. We denote by $\Phi$
the map that sends $I$ to the pair of intervals $(I_r,I_\ell)$.

\figeps{14}{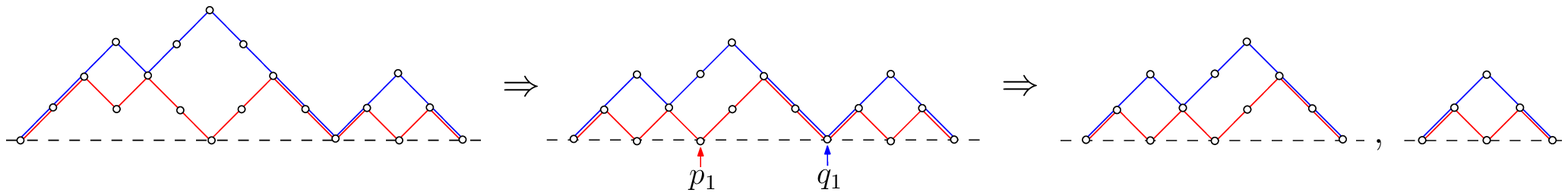}{The recursive decomposition of
  intervals. Starting from an $m$-Tamari interval of size $n$ (here,
  $m=1$ and  $n=7$), one first obtains by reduction an $m$-pointed 
  interval of size $n-1$ (Lemma~\ref{lem:interval_reduc}). This 
  interval is further decomposed  into two intervals, the first
  one being $(m-1)$-pointed.}

Conversely, take an interval $I_r=[\Pr,\Qr]$ and a $k$-pointed interval
$I_\ell=[P^\circ, Q^\circ]$, where 
$P^\circ=(\Pl;p_1,\ldots,p_k)$ and $Q^\circ=(\Ql;q_1,\ldots,q_k)$.
Let $P=\Pl\Pr$, $Q=\Ql\Qr$, and denote by $q_{k+1}$
the point where $\Ql$ and $\Qr$ (and $\Pr$ and $\Pl$) meet.  This is a
contact of $P$ and $Q$.
Then the preimages of $(I_r,I_\ell)$ by $\Phi$
are all the intervals $I=[P^\bullet, Q^\bullet]$ such that
$P^\bullet=(P;p_1,\ldots,p_{k+1})$ 
and $ Q^\bullet=(Q;q_1,\ldots,q_{k+1})$, where $p_{k+1}$ is any \emm active,
contact of $\Pl$. If $\Pl$ has $i$ active contacts and $\Pr$ has $j$ contacts, then  $(I_r,I_\ell)$ has
$i$ preimages, having respectively $j, 1+j, \ldots, i+j-1$ active
contacts 
($j$ active contacts when $p_{k+1}=q_{k+1}$, and $i+j-1$ active contacts when
$p_{k+1}=p_k$). 
Let us write $G^{(k)}(x,y)= \sum_{i\ge 0} G^{(k)}_i(y)
x^i$, so that $G^{(k)}_i(y)$ counts (by the size and the initial rise)
$k$-pointed intervals with $i$ active contacts. The above discussion gives
\begin{eqnarray*}
  G^{(k+1)}(x,y)&=&F(x,1)  \sum_{i\ge 1} G^{(k)}_i(y) (1+x+\cdots + x^{i-1})
\\ 
&=& F(x,1) \sum_{i\ge 1} G^{(k)}_i(y) \frac{x^i -1}{x-1}
\\
&=&F(x,1) \cdot \Delta G^{(k)}(x,y),
\end{eqnarray*}
as claimed in~\eqref{Gk}. The factor $F(x,1)$ accounts for the choice
of $I_r$, and the term $\Delta G^{(k)}(x,y)$ for the choice of
$I_\ell$ and $p_{k+1}$.
This completes the proof of Proposition~\ref{prop:eq}. \qed

\section{Solution of the functional equation}
\label{sec:sol}
In this section, we solve the functional equation of
Proposition~\ref{prop:eq}, and thus establish the main result of this
paper.  We obtain in particular an unexpected symmetry property: the
series $yF^{(m)}(t;x,y)$ is symmetric in $x$ and $y$. In other words, the
joint distribution of the number of non-initial contacts (of the lower
path) and the
initial rise (of the upper path) is symmetric.

For any ring $\GA$, we denote by $\GA[x]$ the ring of polynomials in
$x$ with coefficients in $\GA$, and by $\GA[[x]]$ the ring of \fps\ in
$x$ with coefficients in $\GA$. This notation is extended to the case
of polynomials and series in several indeterminates $x_1, x_2, \ldots$
\begin{Theorem}\label{thm:series}
For  $m\ge 1$,  let $F^{(m)}(t;x,y)$  be the \gf\  of
$m$-Tamari intervals, where $t$ counts the size (divided by $m$), $x$ the
number of contacts of the bottom path, and $y$ the initial rise of the upper
path.  Let $z$, $u$ and $v$ be three indeterminates, and set
\beq\label{t-x-y-param}
t=z(1-z)^{m^2+2m}, 
\quad x=\frac{1+u}{(1+zu)^{m+1}},
\quad \hbox{and} \quad y=\frac{1+v}{(1+zv)^{m+1}}.
\eeq
Then $F^{(m)}(t;x,y)$ becomes a formal power series in $z$ with coefficients
in $\qs[u,v]$,  and this series is rational. More precisely,
\beq\label{F-param}
y F^{(m)}(t;x,y)= \frac{(1+u)(1+zu)(1+v)(1+zv)}{(u-v)(1-zuv)(1-z)^{m+2}}
\left( \frac{1+u}{(1+zu)^{m+1}}-\frac{1+v}{(1+zv)^{m+1}}\right).
\eeq
In particular, $y F^{(m)}(t;x,y)$ is a symmetric series in $x$ and $y$.
\end{Theorem}

\noindent{\bf Remark.} This result was first  guessed for small
values of $m$. 
More precisely, we first guessed the values of $\frac{\partial ^iF}{\partial
    x}(1,1)$ for $0\le i \le m-1$, and then combined  these conjectured
  values with the functional equation to obtain conjectures for
  $F(x,1)$ and $F(x,y)$. 
Let us illustrate our guessing procedure on the case $m=1$. We
first consider the case $y=1$, where the equation reads 
\beq
  F(x,1)= x+xt F(x,1) \frac{F(x,1)-F(1,1)}{x-1}.
\label{eq-func-m1}
\eeq
Our first objective is to guess the value of $F(1,1)$. Using the above
equation, we easily compute, say, the 20 first coefficients
of $F(1,1)$.  Using the {\sc Maple} package {\tt gfun}~\cite{gfun}, we conjecture from this
list of coefficients that $f\equiv F(1,1)$ satisfies
$$
1-16\,t- \left( 1-20\,t \right) f- \left( 3\,t+8\,{t}^{2} \right) {f
}^{2}-3\,{t}^{2}{f}^{3}-{t}^{3}{f}^{4}=0.
$$
Using the package {\tt algcurves}, we find that the above equation
admits a rational parametrization, for instance
$$
t=z(1-z)^3, \quad f=F(1,1)=\frac{1-2z}{(1-z)^3}.
$$
This is the end of the ``guessing'' part\footnote{For a general value
  of $m$, one has to guess the series $\frac{\partial ^iF}{\partial
    x}(1,1)$ for $0\le i \le m-1$. All of them are found to be
  rational functions of $z$, when $t=z(1-z)^{m^2+2m}$.}. Assume the
above identity holds, and
replace $t$ and $F(1,1)$ in~\eqref{eq-func-m1} by their expressions in
terms of $z$. This gives an algebraic equation in $F(x,1)$, $x$ and
$z$. Again, the package {\tt algcurves} reveals that this equation,
seen as an equation in $F(x,1)$ and $x$, has a rational
parametrization, for instance
$$
x=\frac{1+u}{(1+zu)^2}, \quad F(x,1)= {\frac { \left( 1+u \right)  \left( 1-2\,z-{z}^{2}u \right) }{
 \left( 1+zu \right)  \left(1- z \right) ^{3}}}.
$$
Let us finally return to the functional equation defining $F(x,y)$:
$$
 F(x,y)= x+xyt F(x,1) \frac{F(x,y)-F(1,y)}{x-1}.
$$
In this equation, replace $t$, $x$ and $F(x,1)$ by their conjectured
expressions in terms of $z$ and $u$. This gives
\beq\label{m1-zu}
\left( 1+zu -z y \frac{ (1+u ) ^{2}}{u} \right)
 F( x, y) ={\frac {1+u}{1+zu}}-  z y \frac{ (1+u ) ^{2}}{u}F(1,y).
\eeq
We conclude by applying to this equation the \emm kernel method, (see,
e.g.~\cite{hexacephale,bousquet-petkovsek-1,prodinger}): let $U\equiv
U(z;y)$ be the unique \fps\ in $z$ (with coefficients in $\qs[y]$)
satisfying
$$
 U= z y { (1+U ) ^{2}}-zU^2.
$$
Equivalently, 
$$
U= z\frac{1+v}{1-2z-z^2v}, \quad \hbox{with} \quad
y=\frac{1+v}{(1+zv)^2}.
$$
Setting $u=U$ in~\eqref{m1-zu} cancels the left-hand side, and thus
the right-hand side, giving
$$
yF(1,y)= {\frac { \left( 1+v \right)  \left( 1-2\,z-{z}^{2}v \right) }{
\left( 1+zv \right)  \left(1- z \right) ^{3}}}.
$$
A conjecture for the trivariate series $F(t;x,y)$ follows,
using~\eqref{m1-zu}. This conjecture coincides with~\eqref{F-param}. \qee

\medskip
Before we prove Theorem~\ref{thm:series}, let us give a closed form expression
for the number of intervals in  $\cT_{n}^{(m)}$.

\begin{Corollary}\label{prop:number}
Let $m\ge 1$ and $n\ge 1$.  The number of intervals in the Tamari
lattice $\cT_{n}^{(m)}$ is 
$$
f_n^{(m)}=\frac {\mb}{n(mn+1)} {n\mb^2 +m\choose n-1},
$$
where we denote $\mb=m+1$. For $2\le i\le n+1$, the number of intervals in which the bottom path
has $i$ contacts with the $x$-axis is  
\beq\label{enum-double}
f_{n,i}^{(m)}=\frac{(n\mb^2-i\mb+m)! (i\mb-m)!}
{(n\mb^2-n-im+2m)! (n-i+1)! (mi)! (i-2)!}P_m(n,i),
\eeq
where 
$P_m(n,i)$ is a polynomial in $n$ and $i$. In particular,
$$
P_1(n,i)=2, \quad P_2(n,i)=6(33\,in-9\,{i}^{2}+15\,i-2\,n-2).
$$
 More generally,
\begin{multline}\label{pol}
i(i-1) P_m(n,i)=- \mb ! (m-1)! (n-i+1){{i\mb} \choose m}  {{nm(m+2)-im+2m}\choose m-1}
\\
+ \sum_{k=1}^{m-2} k k!^2 (m-k-2)! (m-k-1)! ((i+1)m \mb +2\mb +k)(n-i)(n-i+1) \times\\
{{i\mb -k-1} \choose m-k-1}{im \choose k}
{{n\mb^2-i\mb +m+k}\choose k} {{nm(m+2) -im +2m} \choose  {m-k-2}}
\\+ m!^2{im \choose m-1} \left( i { n\mb^2-i\mb +2m\choose m}
- \frac{(m-1)(i\mb +2)(n-i+1)}m {n\mb^2-i\mb +2m-1\choose m-1}\right).
\end{multline}
\end{Corollary}
\noindent{\bf Remarks}\\
\noindent 1. The case $m=1$ of~\eqref{enum-double} reads
$$
f_{n,i}^{(1)}= \frac {(i-1)(4n-2i +1)!}{(3n-i+2)!(n-i+1)!}{2i\choose i}.
$$
This result can also  obtained using Bernardi and Bonichon's
bijection between intervals of size $n$ in the (usual) Tamari lattice and planar
3-connected triangulations having $n+3$
vertices~\cite{BeBo07}. Indeed, through this 
bijection, the number of contacts in the lower path of
the interval becomes the degree of the root-vertex of the
triangulation, minus one~\cite[Def.~3.2]{BeBo07}. The above result is
thus equivalent to a result of Brown counting triangulations by
the number of vertices and the degree of the
root-vertex~\cite[Eq.~(4.7)]{brown-triang}. 

\medskip
\noindent
2. Our expression of  $P_m$ is not
illuminating, but we have given it to prove that $P_m$ is indeed a
polynomial. If we fix $i$ rather than $m$, then, 
experimentally,
 $P_m(n,i)$ seems to be a sum of two hypergeometric terms in $m$ and
 $n$. More precisely, it appears that 
\begin{multline*}
  P_m(n,i)=\frac {m\mb! (im)!}{  (i\mb-m)! {n\choose {i-1}}} \times\\
\left( \mb R_i(m,n)  {n\mb ^2-(i-2)\mb -1 \choose \mb}
+Q_i(m,n) {nm(m+2)-(i-2)m \choose m}\right),
\end{multline*}
where $R_i$ and $Q_i$ are two polynomials in $m$ and $n$. 
This holds at least for small values of $i$. 

\medskip
\noindent
3. The coefficients of the trivariate series $F(t;x,y)$ do not seem to
have small prime factors, even when $m=1$.

\bigskip

\noindent \emph{Proof of Theorem~{\rm\ref{thm:series}.}}
 The functional equation of
Proposition~\ref{prop:eq} defines a unique formal power series in
$t$ (think of extracting inductively the coefficient of $t^n$ in
$F(t;x,y)$). The coefficients of this series are polynomials in $x$
and $y$.  The parametrized expression of  $ F(t;x,y)$ given in
Theorem~\ref{thm:series}  also defines $ F(t;x,y)$ uniquely as a power
series in $t$, because~\eqref{t-x-y-param}
defines $z$, $u$ and $v$ uniquely as formal power series in $t$ (with
coefficients in $\qs$, $\qs[x]$ and $\qs[y]$ respectively). 
Thus it suffices to prove  that the series $F(t;x,y)$
of Theorem~\ref{thm:series} satisfies the equation of
Proposition~\ref{prop:eq}. 

If $G(t;x,y)\equiv G(x,y)$ is any series in $\qs[x,y][[t]]$, then
performing the change of variables~\eqref{t-x-y-param} gives
  $G(t;x,y)= H(z;u,v)$, where 
$$
H(z;u,v)\equiv H(u,v)= G\left(z(1-z)^{m^2+2m}; \frac
  {1+u}{(1+zu)^{m+1}}, \frac {1+v}{(1+zv)^{m+1}}\right).
$$
Moreover, if $F(x,y)$ is given by~\eqref{F-param}, then 
$$
F(x,1)= \frac{(1+u)(1+zu)}{u(1-z)^{m+2}}
\left( \frac{1+u}{(1+zu)^{m+1}}-1\right),
$$
and
$$
F(x,1)\Delta G(x,y)
=
\frac{(1+u)(1+zu)}{(1-z)^{m+2}} \frac{H(u,v)-H(0,v)}u.
$$
Let us define an operator $ \Lambda$ as follows: for any series $H(z;u,v)\in \qs[u,v][[z]]$,
\beq\label{newop}
\Lambda H(z;u,v):= (1+u)(1+zu) \frac{H(z;u,v)-H(z;0,v)}u.
\eeq
Then the series $F(t;x,y)$ of Theorem~\ref{thm:series} satisfies the equation of
Proposition~\ref{prop:eq} if and only if   
 the series $H(u,v)$ obtained by performing the change of
variables~\eqref{t-x-y-param} in  $y(1-z)^{m+2}F(x,y)$, that is,
\beq\label{H-def}
H(u,v)=\frac{(1+u)(1+zu)(1+v)(1+zv)}{(u-v)(1-zuv)}
\left( \frac{1+u}{(1+zu)^{m+1}}-\frac{1+v}{(1+zv)^{m+1}}\right).
\eeq
satisfies
\beq
  \label{final-id}
z\Lambda^{(m)} H(u,v)
=
\frac{(1+zu)^{m+1}(1+zv)^{m+1}}{(1+u)(1+v)} H(u,v)
-(1-z)^{m+2}.
\eeq
Hence we simply have to prove an identity on rational functions. 
Observe that both $H(u,v)$ and the conjectured expression of
$\Lambda^{(m)} H(u,v)$ are symmetric in $u$ and $v$. More generally,
computing (with the help of {\sc Maple}) the rational functions
$\Lambda^{(k)} H(u,v)$  for a few values of $m$ and $k$ suggests that
these fractions are \emm always, symmetric in $u$ and $v$. This
observation raises the following question: Given a symmetric function
$H(u,v)$, when is $\Lambda H(u,v)$ also symmetric? This leads to the 
following lemma, which will reduce the proof of~\eqref{final-id} to
the case $v=0$.

\begin{Lemma}\label{lem:sym}
  Let $H(z;u,v)\equiv H(u,v)$ be a series of $\qs[u,v][[z]]$,
  symmetric in $u$ and $v$. Let $\Lambda$ be the operator defined by~\eqref{newop}, and denote $H_1(u,v):=
\Lambda H(u,v)$. Then  $H_1(u,v)$ is  symmetric in $u$ and $v$ if
and only if $H$ satisfies
\beq\label{sym-form}
H(u,v)= \frac { u(1+v)(1+zv)
  H(u,0)-v(1+u)(1+zu)H(v,0)}{(u-v)(1-zuv)}.
\eeq
If this holds,  then $H_1(u,v)$ also
 satisfies~\eqref{sym-form} (with $H$ replaced by $H_1$).
By induction, the same holds for $H_k(u,v):=\Lambda^{(k)} H(u,v)$.
\end{Lemma}
\noindent The proof is a straightforward calculation.\qed

\medskip

Note that a series $H$ satisfying~\eqref{sym-form} is characterized by
the value of $H(u,0)$.   The series $H(u,v)$ given by~\eqref{H-def}
satisfies~\eqref{sym-form}, with
$$
H(u,0)= \frac{(1+u)(1+zu)}u\left(
{ \frac{1+u}{(1+zu)^{m+1}}-1} \right) = \Lambda \left(
  \frac{1+u}{(1+zu)^{m+1}}\right).
$$
Moreover, one easily checks that the right-hand side of~\eqref{final-id}  also
satisfies~\eqref{sym-form}, as expected from
Lemma~\ref{lem:sym}. Thus it suffices to prove the case $v=0$
of~\eqref{final-id}, namely
\beq\label{final-id-simple}
z\Lambda^{(m+1)} \left( \frac{1+u}{(1+zu)^{m+1}}\right)
=
\frac{(1+u)(1+zu)}{u}
\left(1- \frac{(1+zu)^{m+1}}{1+u}\right)
-(1-z)^{m+2}.
\eeq
This will be a simple consequence of the following lemma.

\begin{Lemma}\label{lem:id}
Let $\Lambda$ be the operator defined by~\eqref{newop}. For 
$m\ge 1$,
$$
\Lambda^{(m)} \left( \frac 1{(1+zu)^{m}}\right)=(1-z)^m -(1+zu)^m.
$$
\end{Lemma}
\begin{proof}
We will actually prove a more general identity. Let  $1\le k \le m$,
and   denote $w=1+zu$. Then 
  \begin{multline}\label{identity}
    \Lambda^{(k)} \left( \frac1{(1+  zu)^{m}}\right)
=
 \frac{(1-z)^k} {w^{m-k}}-\sum_{i=k}^{m-1} \sum_{j=1}^k \frac {(-1)^{k+j} z^{k-j+1}}{w^{m-i-1}} 
  {k\choose j-1} {i-j+1\choose k-j}
 \\ 
+\sum_{i=1}^{k-1}\sum_{j=1}^i (-1)^{j-1}z^jw^{k-i} {i-1 \choose
j-1} {m-k+j-1 \choose j}-w^k.
\end{multline}
The case $k=m$ is the identity of Lemma~\ref{lem:id}.
In order to prove~\eqref{identity}, we   need an expression of
$\Lambda( w^p)$, for all $p\in \zs$. 
Using the definition~\eqref{newop} of $\Lambda$, one obtains, for $p\ge 1$,
\beq\label{lambda-elem}
\left\{
\begin{array}{lll}
 \displaystyle  \Lambda \left( \frac 1 {w^p}\right) &=& \displaystyle 
\frac{1-z}{w^{p-1}} - z\sum_{a=0}^{p-2} \frac 1{w^a} -w,
\\
\displaystyle \Lambda(1)&=&0,
\\
\displaystyle \Lambda \left(  {w^p}\right) &=& \displaystyle (z-1) w+z \sum_{a=2}^p w^a +w^{p+1}.
\end{array}\right.
\eeq

We now prove~\eqref{identity}, by induction on $k\ge 1$.  For $k=1$, \eqref{identity} coincides with the expression of
 $\Lambda(1/w^p)$ given above (with $p$ replaced by $m$).
Now let $1\le k<m$. Apply $\Lambda$
to~\eqref{identity}, use~\eqref{lambda-elem} to express the terms
$\Lambda(w^p)$ that appear, and then check that the coefficient of
$w^az^b$ is what it is expected to be, for all values of $a$ and
$b$. The details are a bit tedious, but elementary. One needs to apply
a few times the following identity:
$$
\sum_{r=r_1}^{r_2} {r-a \choose b}
={\frac { \left(r_2+1- a-b \right) }{b+1}}{r_2+1-a\choose b}-{\frac { \left(r_1- a-b \right) }{b+1}}{r_1-a\choose b}
.
$$

We give in  the appendix a constructive proof of Lemma~\ref{lem:id},
which does not require to guess the more general
identity~\eqref{identity}.
It is also possible to derive~\eqref{identity} combinatorially
from~\eqref{lambda-elem} using one-dimensional lattice paths (in this
setting,~\eqref{lambda-elem} describes what steps are allowed if one
starts at position $p$, for any $p \in \zs$).
\end{proof}
Let us now return to the proof of~\eqref{final-id-simple}. We write
$$
z \frac{1+u}{(1+uz)^{m+1}}= \frac 1 {(1+uz)^m}+\frac
{z-1}{(1+uz)^{m+1}}.
$$
Thus 
\begin{eqnarray*}
  z \Lambda^{(m+1)}
\left(\frac{1+u}{(1+uz)^{m+1}}\right)&= &
\Lambda\left( \Lambda^{(m)} \left(\frac 1 {(1+u z)^m}\right)\right)
+(z-1) \Lambda^{(m+1)} \left( \frac 1{(1+uz)^{m+1}}\right)
\\
&=&
\Lambda\left((1-z)^m -(1+uz)^m\right)+(z-1) \left((1-z)^{m+1}
  -(1+uz)^{m+1}\right) 
\end{eqnarray*}
by Lemma~\ref{lem:id}. Eq.~\eqref{final-id-simple} follows, and
Theorem~\ref{thm:series} is proved.
\qed

\bigskip
\noindent {\emph {Proof of Corollary~{\rm\ref{prop:number}.}}}
\noindent
Let us first determine the coefficients of $F(t;1,1)$. By letting $u$
and $v$ tend to $0$ in the expression of $yF(t;x,y)$, we obtain
$$
F(t;1,1)=\frac{1-(m+1)z}{(1-z)^{m+2}},
$$
where $t=z(1-z)^{m^2+2m}$. The Lagrange inversion formula gives
$$
[t^n] F(t;1,1)= \frac 1 n [t^{n-1 }]
  \frac{1-(m+1)^2t}{(1-t)^{nm(m+2)+m+3}},
$$
and the expression of $f_n^{(m)}$ follows after an elementary
coefficient extraction.

\medskip
We now wish to express the coefficient of $t^nx^i$ in
$$
F(t;x,1)= \frac{(1+u)(1+zu)}{u(1-z)^{m+2}}
\left( \frac{1+u}{(1+zu)^{m+1}}-1\right).
$$
We will expand this series, first in $x$, then in $t$, applying  the Lagrange 
 inversion formula first to $u$, then to $z$.
We first expand $(1-z)^{m+2}F(t;x,1)$ in partial fractions of $u$:
$$
(1-z)^{m+2}F(t;x,1)=
-z\chi_{m>1}-(1+zu)-\sum_{k=1}^{m-2}
\frac{z}{(1+uz)^k}+\frac{1-z^2}{z(1+uz)^{m-1}} 
-\frac{(1-z)^2}{z(1+uz)^{m}}.
$$
By the Lagrange inversion formula, applied to $u$, we have, for $i\ge
1$ and $p\ge -m$, 
$$
[x^i] (1+zu)^p= \frac p i {{i\mb+p-1}\choose {i-1}} z^i(1-z)^{im+p},
$$
with $\mb=m+1$. Hence, for $i\ge 1$,
\begin{multline*}
 i [x^i] F(t;x,1)=- {i\mb\choose i-1} z^i (1-z)^{(i-1)m-1}
+\sum_{k=1}^{m-2}k {{i\mb-k-1}\choose {i-1}} z^{i+1} (1-z)^{(i-1)m-k-2}\\
- (m-1) {{i\mb-m}\choose i-1}z^{i-1}(1+z) (1-z)^{(i-2)m}
+m{{(i-1)\mb}\choose {i-1}} z^{i-1}(1-z)^{(i-2)m}.
\end{multline*}
We rewrite the above line as
$$
{{i\mb-m}\choose {i-1}}\left( \frac i{i\mb-m} z^{i-1} (1-z)^{(i-2)m}
-(m-1) z^i  (1-z)^{(i-2)m}\right).
$$
Recall that $z=\frac t{(1-z)^{m^2+2m}}$. Hence, for $i\ge 1$,
\begin{multline*}
i  [x^it^n] F(t;x,1)=- {i\mb \choose i-1} [t^{n-i}]
  \frac 1{(1-z)^{\mb (im+1)}}
\\+\sum_{k=1}^{m-2} k
  {{i\mb-k-1}\choose {i-1}}[t^{n-i-1}] \frac
  1{(1-z)^{(i+1)m\mb +2\mb+k}}
\\
+{{i\mb-m}\choose i-1} \left( 
\frac i{i\mb-m} [t^{n-i+1}]\frac 1{(1-z)^{m(i\mb-m)}}
- (m-1) [t^{n-i}] \frac 1{(1-z)^{m(i\mb +2)}}
\right).
\end{multline*}
By the Lagrange inversion formula, applied to $z$, we have, for $p\ge
1$ and $n\ge 1$,
$$
[t^n] \frac 1{(1-z)^p}=\frac p n {n\mb^2+p-1\choose {n-1}}.
$$
This formula actually holds for  $n =0$ if we write it as
$$
[t^n] \frac 1{(1-z)^p}=\frac{p\,(n\mb^2+p-1)! }{n!\,(n\mb^2 -n +p)!},
$$
and actually for $n<0$ as well with the convention ${a \choose {n-1
  }}=0$ if $n<0$. With this convention, we have,  for $ 1\le i \le n+1$,
\begin{multline*}
 i [x^it^n] F(t;x,1)=-\frac{\mb (im+1)}{n-i} {i\mb\choose i-1}
  {{n\mb^2-i\mb+m}\choose {n-i-1}}
\\+\sum_{k=1}^{m-2} k\frac{(i+1)m\mb+2\mb+k}{n-i-1}
{{i\mb-k-1}\choose {i-1}}
{{n\mb^2-i\mb+m+k}\choose {n-i-2}}
\\+ m{{i\mb-m}\choose {i-1}} \left(
  \frac{i}{n-i+1}{{n\mb^2-i\mb+2m}\choose {n-i}}
- (m-1) \frac {i\mb+2}{n-i}{{ n\mb^2-i\mb+2m-1}\choose n-i-1}\right).
\end{multline*}
This gives the expression~\eqref{enum-double} of $f_{n,i}^{(m)}$, with
$P_m(n,i)$ given by~\eqref{pol}. Clearly, $i(i-1) P_m(n,i)$ is a
polynomial in $n$ and $i$, but  we still have to prove that it is
divisible by $i(i-1)$.

For $m\ge 1$ and $1\le k\le m-2$, the polynomials ${i\mb \choose m}$ and $
{im \choose k}$ are  divisible by $i$. The next-to-last term
of~\eqref{pol} contains an explicit factor $i$. The last term vanishes
if $m=1$, and otherwise contains a factor  ${im  \choose m-1}$, which
is a multiple of $i$. Hence each term of~\eqref{pol} is divisible by
$i$.

Finally, the right-hand side of~\eqref{pol} is easily evaluated to be
0 when $i=1$, using the {\tt sum} function of {\sc Maple.}
\qed

\section{Final comments}
\label{sec:final}
\noindent {\bf Bijective proofs?}
Given the simplicity of the numbers~\eqref{number}, it is natural to ask
about a bijective enumeration of $m$-Tamari intervals. A related
question would be  to extend  the bijection of~\cite{BeBo07}
(which transforms 1-Tamari 
intervals into triangulations)
into a bijection between $m$-Tamari  intervals  and
certain maps (or related structures, like \emm balanced trees, or
\emm mobiles,~\cite{Sch97,bouttier-mobiles}).  Counting these
structures in a bijective way (as is done in~\cite{poulalhon-schaeffer-triang} for triangulations)
would then provide a bijective proof  of~\eqref{number}. 

\medskip
\noindent{\bf Symmetry.}
The fact that the joint distribution of the number of non-initial contacts of the
lower path and the initial rise of the upper path is symmetric remains
a combinatorial mystery to us, even when $m=1$. What \emm is, easy to see
is that the joint distribution
of the number of non-initial contacts of the 
lower path and the \emm final descent, of the upper path is symmetric.
Indeed, there exists a simple involution on Dyck paths that
reverses the Tamari order and exchanges these two parameters: If
we consider Dyck paths as postorder encodings of binary trees,  this
involution amounts to a simple reflection of trees.
Via the bijection of~\cite{BeBo07}, these two parameters  correspond to the degrees of two vertices of
the root-face of the triangulation~\cite[Def.~3.2]{BeBo07}, so that
the symmetry is also clear in this setting.

\medskip
\noindent{\bf A $q$-analogue of the functional equation.} As described in the introduction, the numbers $f_n^{(m)}$ are
conjectured to give the dimension  of certain polynomial rings $\DR_{3
  ,n}^{m \ \varepsilon}$. These rings are tri-graded (with respect to
the sets of variables 
$\{x_i\}$, $\{y_i\}$ and $\{z_i\}$), 
 and it is conjectured~\cite{bergeron-preville}
that the dimension of the homogeneous component in the $x_i$'s of degree $k$ 
is the number of intervals $[P,Q]$ in $\cT_{n}^{(m)}$ such that the
longest chain from $P$ to $Q$, in the Tamari order, has length
$k$. One can recycle the recursive description of intervals described
in Section~\ref{sec:rec} to generalize the functional equation of
Proposition~\ref{prop:eq}, taking into account (with a new variable $q$) this
distance. Eq.~\eqref{eq:Fb} remains valid, upon defining the operator
$\Delta$ by
$$
\Delta S(x)=\frac{S(qx)-S(1)}{qx-1}.
$$
The coefficient of $t^n$ in the series $F(t,q;x,y)$ does not seem to
factor, even when $x=y=1$. The coefficients of the bivariate series
$F(t,q;1,1)$ have large prime factors.

\medskip
\noindent{\bf More on $m$-Tamari lattices?}
We do not know of any
simple description of the $m$-Tamari lattice in terms of rotations in
$m+1$-ary trees (which are equinumerous with $m$-Dyck paths). 
A rotation for ternary trees is defined 
in~\cite{pallo}, but does not give a lattice.
However, as noted by the referee, if we interpret $m$-ballot paths as
the \emm prefix, (rather than postfix) code of an $m+1$-ary tree, the
covering relation can be described quite simply. One first chooses a leaf $\ell$
that is followed (in prefix order) by an internal node  $v$. Then,
denoting by $T_0, \ldots, T_m$ the $m+1$ subtrees attached to $v$, from left to
right, we insert $v$ and its first $m$ subtrees 
in place of the leaf $\ell$, which becomes the rightmost child of $v$.
 The rightmost subtree of $v$, $T_m$,
finally takes the former place of $v$ (Figure~\ref{fig:m-flip}).
\figeps{9}{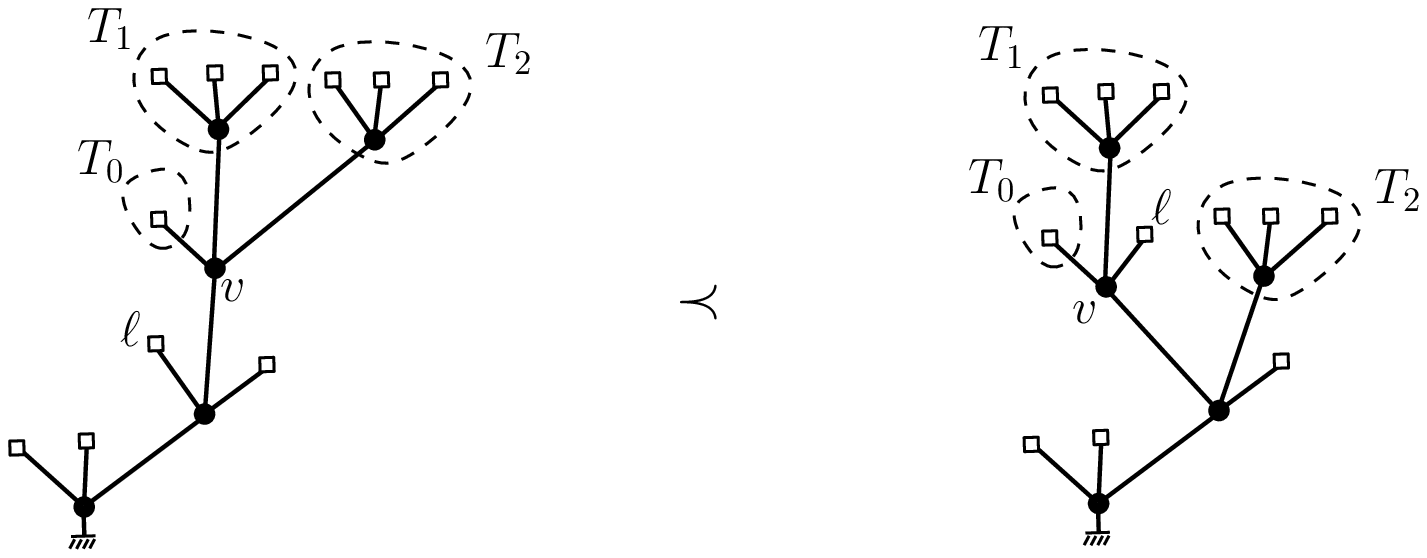}{The covering relation of Figure~\ref{fig:push_mWalk} translated
  in terms of  ternary trees.}

More generally, it may be worth exploring analogues for the  $m$-Tamari
lattices of the numerous questions that have been studied for the
usual Tamari lattice. To mention only one, what is the \emm diameter, of the
$m$-Tamari lattice, that is,  the maximal distance between two
$m$-Dyck paths in the Hasse diagram? When $m=1$, it is known to be
$2n-6$ for $n$ large enough, but the proof is as complicated as the
formula is simple~\cite{dehornoy,sleator}. 

\spacebreak

\bigskip
\noindent{\bf Acknowledgements.} We are grateful to François Bergeron
for advertising in his lectures the conjectural interpretation of the
numbers~\eqref{number} in terms of Tamari intervals. We also thank Gwendal Collet and Gilles Schaeffer   for
interesting discussions on this topic.

\bibliographystyle{plain}
\bibliography{tamar.bib}

\spacebreak

\appendix
\medskip
\noindent{\bf Appendix. A constructive approach to Lemma~\ref{lem:id}.} 
In order to prove Lemma~\ref{lem:id}, we had to prove the more general
identity~\eqref{identity}.  This identity was first
guessed by expanding $\Lambda^{(k)}(1/w^m)$ in $w$ and $z$, for
several values of $k$ and $m$. Fortunately, the coefficients in this
expansion turned out to be simple products of binomial coefficients. 

 What if these coefficients had not been so simple? 
A  constructive approach goes as follows. Introduce the  following two
\fps\ in\footnote{The variable $t$ that we use here has nothing to do
  with the variable $t$ that occurs in the \gf\ $F(t;x,y)$ of
  intervals.}  $t$ and $s$, with coefficients in $\qs[w,1/w,z]$:
$$
P(t;s)= \sum_{m\ge 1, k\ge 0} t^k s^{m-1} \Lambda^{(k)}(w^m)
\quad \hbox{and} \quad 
N(t;s)= \sum_{m\ge 0, k\ge 0} t^k s^{m} \Lambda^{(k)}\left(\frac
  1{w^m}\right),
$$
where we still denote $w=1+zu$. Observe that 
$$
P(t;0)= \sum_{ k\ge 0} t^k  \Lambda^{(k)}(w).
$$
We want to compute the coefficient of $t^m s^m$ of $N(t;s)$, since
this coefficient  is $\Lambda^{(m)}(1/w^m)$.

Eq.~\eqref{lambda-elem} yield
functional equations for the series $P$ and $N$. For $P(t;s)$ first,
\begin{eqnarray*}
  P(t;s) &=& \sum_{m\ge 1} s^{m-1} w^m + t 
 \sum_{m\ge 1, k\ge 1} t^{k-1} s^{m-1} \Lambda^{(k-1)}
\left( (z-1) w+z \sum_{a=2}^m w^a +w^{m+1}\right)
\\
&=&\frac w{1-sw} + \frac{t(z-1)}{1-s} P(t;0)+ \frac {tz}{1-s}
\left(P(t;s)-P(t;0)\right) + t\, \frac{P(t;s)-P(t;0)}s.
\end{eqnarray*}
Equivalently,
\beq\label{eq:P}
\left( 1 - \frac{tz}{1-s}-\frac t s\right) P(t;s)= \frac w{1-sw} -
\frac{t P(t;0)}{s(1-s)}
.
\eeq
Now for $N(t;s)$, we have
\begin{eqnarray*}
  N(t;s) &=& \sum_{m\ge 0}  \frac {s^{m}}{w^m} + t 
 \sum_{m\ge 1, k\ge 1} t^{k-1} s^{m} \Lambda^{(k-1)}
\left(\frac{1-z}{w^{m-1}} - z\sum_{a=0}^{m-2} \frac 1{w^a} -w \right)
\\
&=&\frac 1{1-s/w} +ts(1-z)N(t;s)-\frac{tzs^2}{1-s} N(t;s)
-\frac{ts}{1-s} P(t;0) .
\end{eqnarray*}
Equivalently,
\beq\label{eq:N}
\left( 1-ts + \frac{tzs}{1-s}\right) N(t;s)= \frac
1{1-s/w}-\frac{ts}{1-s} P(t;0) .
\eeq
Equation~\eqref{eq:P} can be solved using the kernel method (see
e.g.~\cite{hexacephale,bousquet-petkovsek-1,prodinger}):
let $S\equiv S(t,z)$ be the unique \fps \ in $t$, with coefficients in
$\qs[z]$, having constant term 0 and satisfying 
$$ 1 - \frac{tz}{1-S}-\frac t S=0.$$
That is,
\beq\label{S-def}
S= \frac{1+t-tz-\sqrt{1-2t(1+z)+t^2(1-z)^2}}2.
\eeq
Then setting $s=S$ cancels the left-hand side of~\eqref{eq:P}, giving 
$$
P(t;0)= \frac{wS(1-S)}{t(1-wS)}.
$$
Combined with~\eqref{eq:N}, this yields an explicit expression of
$N(t;s)$:
 $$
N(t;s)= \frac 1{1-ts + \frac{tzs}{1-s}} \left( \frac
1{1-s/w} - \frac{ws S(1-S)}{(1-s)(1-wS)}\right).
$$
We want to extract from this series
the coefficient of $t^m s^m$, and obtain the simple expression
$(1-z)^m-w^m$ predicted by Lemma~\ref{lem:id}. Clearly, the first part of the above 
expression of $N(t;s)$ (with non-positive powers of $w$) contributes $(1-z)^m$,
as expected. For $i\ge 1$, the coefficient of $w^i$ in
 the second part of $N(t;s)$ is
$$
R_i:=- \frac{s S^i(1-S)}{(1-s)\left(1-ts + \frac{tzs}{1-s}\right)}.
$$
Recall that $S$, given by~\eqref{S-def}, depends on $t$ and $z$, but not on $s$. 
Since $S=t+O(t^2)$, the coefficient of $t^m s^m$ in $R_i$  is
zero for $i>m$. When $i=m$, it is easily seen to be $-1$, as expected. 
In order to prove that the coefficient of $t^m s^m$ in $R_i$ is zero
when $0<i<m$, we first perform a partial fraction expansion of $R_i$ in $s$,
using
$$
(1-s)\left(1-ts + \frac{tzs}{1-s}\right)= (1-sS)(1-st/S),
$$
where $S$ is defined by~\eqref{S-def}.
This gives
$$
R_i= -\frac{S^{i+1}(1-S)}{t-S^2}\left( \frac 1{1-ts/S} - \frac 1 {1-sS}\right),
$$
so that
$$
[s^m] R_i= -\frac{S^{i+1-m}(1-S)}{t-S^2}\left(t^m-S^{2m}\right)
= \sum_{j=0}^{m-1} t^{m-1-j} S^{2j+i-m+1}(S-1)
$$
and
\beq\label{smtm}
[s^mt^m] R_i
= \sum_{j=0}^{m-1} [t^{j+1}] S^{2j+i-m+1}(S-1)= \sum_{j=0}^{m-i} [t^{j+1}] S^{2j+i-m+1}(S-1).
\eeq
  The Lagrange inversion gives, for $n\ge 1$ and $k\in \zs$,
$$
[t^n] S^k(S-1)= 
\left\{
\begin{array}{lll}
 0 &&\hbox{if } n<k;
\\
 -1 &&\hbox{if } n=k;
\\
1-kz && \hbox{if } n=k+1;
\\
\displaystyle \frac 1 n \sum_{p=1}^{n-k} z^p {n \choose p} {{n-k-1}\choose {p-1}} \frac{n-p-kp}{n-k-1} && \hbox{otherwise.}
\end{array}
\right.
$$
Returning to~\eqref{smtm}, this gives
$$
[s^mt^m] R_i= -(m-i-1)z + \sum_{j=0}^{m-i-2}\ \sum_{p=1}^{m-i-j}\frac { z^p}{j+1}  {{j+1} \choose p} {{m-i-j-1} \choose {p-1}} \frac{j+1-p(2j+i-m+2)}{m-i-j-1}.
$$
Proving that this is zero boils down to proving, that, for $1\le p \le m-i$,
$$
 \sum_{j=0}^{m-i-2}  \frac 1{j+1} {{j+1} \choose p} {{m-i-j-1} \choose {p-1}} \frac{j+1-p(2j+i-m+2)}{m-i-j-1}= (m-i-1) \chi_{p=1}.
$$
This is  easily proved using Zeilberger's
algorithm~\cite[Chap.~6]{AB}, 
via the {\sc Maple} package  {\sc Ekhad} (command {\tt zeil}),
or directly using the {\sc Maple} command {\tt sum}.

\end{document}